\documentclass[submission,copyright,creativecommons]{eptcs}
\usepackage[hypcap]{caption}

\usepackage{amsmath}
\usepackage{amssymb}
\usepackage{latexsym}
\usepackage{makeidx}
\usepackage{amsthm}
\usepackage[inline]{enumitem}
\usepackage{stmaryrd}
\usepackage{tikzexternal}
\usepackage{multirow,bigdelim}
\usepackage{emptypage}
\newcommand{\p}{\!:\!}
\newcommand{\Sh}{\mathrm{Sh}}

\newcommand{\isoarrow}{\xrightarrow{\;\sim\;}}

\newcommand{\op}[1]{{#1}^{op}}

\newcommand{\Diag}{\mathrm{Diag}}

\newcommand{\true}{\mathrm{true}}

\newcommand{\catset}{\mathbf{Set}}

\newcommand{\N}{\mathbf{N}}
\newcommand{\X}{\mathbf{X}}
\newcommand{\Y}{\mathbf{Y}}
\newcommand{\F}{\mathbf{F}}
\newcommand{\y}{\mathbf{y}}
\newcommand{\bP}{\mathbf{P}}
\newcommand{\Q}{\mathbf{\widetilde{P}}}
\newcommand{\sN}{\widetilde{\N}}
\newcommand{\E}{\mathbf{E}}
\newcommand{\G}{\mathbf{G}}

\newcommand{\mc}[1]{\mathcal{#1}}

\newcommand{\ra}{{\mc{RA}_K}}

\newcommand{\raop}{{\op{\mc{RA}}_K}}
\newcommand{\raqop}{{\op{\mc{RA}}_\rationals}}

\newcommand{\nats}{\mathbb{N}}

\newcommand{\rationals}{\mathbb{Q}}

\newcommand{\ideal}[1]{\langle #1 \rangle}

\newcommand{\dom}[1]{\textup{dom(}#1\textup{)}}
\newcommand{\codom}[1]{\textup{cod(}#1\textup{)}}
\newcommand{\forces}{\Vdash}

\newcommand{\bJ}{\mathbf{J}}

\newtheorem{thm}{Theorem}[section]
\newtheorem{lem}[thm]{Lemma}
\newtheorem{fact}[thm]{Fact}

\newtheorem{cor}[thm]{Corollary}
\theoremstyle{definition}
\newtheorem{defn}[thm]{Definition}

\title{A Sheaf Model of the Algebraic Closure}
\author{Bassel Mannaa \qquad\qquad Thierry Coquand
\institute{Department of Computer Science and Engineering\\
University of Gothenburg\thanks{The research leading to this work has been supported by ERC Advanced grant project 247219.}\\
Gothenburg, Sweden}
\email{bassel.mannaa@cse.gu.se \quad \qquad thierry.coquand@cse.gu.se}}

\def\titlerunning{A Sheaf Model of the Algebraic Closure}
\def\authorrunning{B. Mannaa \& T. Coquand}

\begin{document}
\maketitle

%%%%%%%%%%%%%%%%%%%%%%%%%%%%%ABSTRACT%%%%%%%%%%%%%%%%%%%%%%%%%%%%%%%%
\begin{abstract}
In constructive algebra one cannot in general decide the irreducibility of a polynomial over a field $K$. This poses some problems to showing the existence of the algebraic closure of $K$. We give a possible constructive interpretation of the existence of the algebraic closure of a field in characteristic $0$ by building, in a constructive metatheory, a suitable site model where there is such an algebraic closure. One can then extract computational content from this model. We give examples of computation based on this model.
\end{abstract}

%%%%%%%%%%%%%%%%%%%%%%%%%%%INTRODUCTION%%%%%%%%%%%%%%%%%%%%%%%%%%%%%%%%
\section{Introduction}

Since in general it is not decidable whether a given polynomial over a field is irreducible, even when the field is given explicitly \cite{Frohlich26011956}, the notion of algebraic field extension and consequently the notion of algebraic closure becomes problematic from a constructive point of view. Even in situations where one can constructively assert the existence of an algebraic closure of a field \cite[Ch. 6]{Mines} the computational content of such assertions are not always clear. We present a constructive interpretation of the algebraic closure of field $K$ in characteristic $0$ as a site model. Our approach is different from \cite{scedrov1984forcing} in that we do not assume a polynomial over a field to be decomposable into irreducible factors. The model presented here has a direct computational content and can be viewed as a model of dynamical evaluation in the sense of Duval \cite{duval0} (see also \cite{Coste2001203}). The site, described in section \ref{site}, is given by the category of finitely presented (von Neumann) regular algebras over $K$ with the appropriate Grothendieck topology. In section \ref{algclosed} we prove that the topos $\mc{E}$ of sheaves on this site contains a model of an algebraically closed field extension of $K$. An alternative approach using profinite Galois group is presented in \cite{Kennison19827}.  We also investigate some of the properties of the topos $\mc{E}$. Theorem \ref{choicefails} shows that the axiom of choice fails to hold in $\mc{E}$ whenever $K$ is not algebraically closed. Theorem \ref{dependentchoicefails} shows that when the base field $K$ is the rationals the weaker axiom of dependent choice fails to hold. We restrict ourselves to constructive metatheory throughout the paper with the exception of section \ref{sec:booleaness} in which we show that in a classical metatheory the topos $\mc{E}$ is boolean (Theorem \ref{booleaness}). As we will demonstrate by Theorem \ref{notshownboolean} this cannot be shown to hold in an intuitionistic metatheory.

\section{Coverage, sheaves, and Kripke--Joyal semantics}
\label{prel}
In this section we recall some notions that we will use in the remainder the paper, mostly following the presentation in \cite{elephant2}. A \emph{coverage} on a category $\mc{C}$ is a function $\bJ$ assigning to each object $C$ of $\mc{C}$ a collection $\bJ(C)$ of families of morphisms with codomain $C$ such that for any $\{f_i:C_i \rightarrow C\}_{i\in I} \in \bJ(C)$ and morphism $g:D\rightarrow C$ of $\mc{C}$ there exist $\{h_j: D_j \rightarrow D\}_{j\in J} \in \bJ(D)$ such that for each $j \in J$ the morphism $g h_j$ factors through $f_\ell$ for some $\ell \in I$. A family $S \in \bJ(C)$ is called \emph{elementary cover} or elementary covering family of $C$. A site is a category with coverage $(\mc{C},\bJ)$. For a presheaf $\bP:\op{\mc{C}} \rightarrow \catset$ and family $S=\{g_i:A_i \rightarrow A\}_{i\in I}$ of morphisms of $\mc{C}$ we say that a family $\{s_i\in \bP(A_i)\}_{i\in I}$ is compatible if for each $\ell, j \in I$ whenever we have $h:B\rightarrow A_\ell$ and $f:B \rightarrow A_j$ such that $g_\ell h = g_j f$ then $s_\ell h = s_j f$, where by $s_\ell h$ we mean the restriction of $s_\ell$ along $h$, i.e. $\bP(h) s_\ell$. A presheaf $\bP$ is a sheaf on the site $(\mc{C},\bJ)$ if for any object $C$ and any $\{f_i:C_i\rightarrow C\}_{i\in I}\in \bJ(C)$ if $\{s_i\in \bP(C_i)\}_{i\in I}$ is compatible then there exist a unique $s \in \bP(C)$ such that $s f_i = s_i$. We call such $s$ the amalgamation of $\{s_i\}_{i\in I}$. Let $\bJ$ be a coverage on $\mc{C}$ we define a closure $\bJ^\ast$ of $\bJ$ as follows: For all objects $C$ of $\mc{C}$ \begin{enumerate*}[label=\roman*.] \item $\{C \xrightarrow{1_C} C\} \in \bJ^\ast(C)$, \item If $S\in \bJ(C)$ then $S \in \bJ^\ast(C)$, and, \item If $\{C_i \xrightarrow{f_i} C\}_{i\in I} \in \bJ^\ast(C)$ and for each $i\in I$, $\{C_{ij}\xrightarrow{g_{ij}} C_i\}_{j\in J_i} \in \bJ^\ast(C_i)$ then $\{C_{ij}\xrightarrow{f_i g_{ij}} C\}_{i\in I, j \in J_i} \in \bJ^\ast(C)$.\end{enumerate*} A family $T \in \bJ^\ast(C)$ is called \emph{cover} or covering family of $C$.

We work with a typed language with equality $\mc{L}[V_1,...,V_n]$ having the basic types $V_1,...,V_n$ and type formers $-\times-, (-)^-, \mc{P}(-)$. The language $\mc{L}[V_1,...,V_n]$ has typed constants and function symbols. For any type $Y$ one has a stock of variables $y_1,y_2,...$ of type $Y$. Terms and formulas of the language are defined as usual. We work within the proof theory of intuitionistic higher-order logic (IHOL). A detailed description of this deduction system is given in \cite{awodey97logicin}. 

The language $\mc{L}[V_1,...,V_n]$ along with deduction system IHOL can be interpreted in an elementary topos in what is referred to as \emph{topos semantics}. For a sheaf topos this interpretation takes a simpler form reminiscent of Beth semantics, usually referred to as \emph{Kripke--Joyal sheaf semantics}. We describe this semantics here briefly following \cite{scedrov1984forcing}.

Let $\mc{E}=\Sh(\mc{C},\bJ)$ be a sheaf topos. An interpretation of the language $\mc{L}[V_1,...,V_n]$ in the topos $\mc{E}$ is given as follows:  Associate to each basic type $V_i$ of $\mc{L}[V_1,...,V_n]$ an object $\mathbf{V}_i$ of $\mc{E}$. If $Y$ and $Z$ are types of $\mc{L}[V_1,...,V_n]$ interpreted by objects $\mathbf{Y}$ and $\mathbf{Z}$, respectively, then the types $Y\times Z, Y^Z, \mc{P}(Z)$ are interpreted by $\mathbf{Y} \times \mathbf{Z}, \mathbf{Y}^{\mathbf{Z}}, \Omega^{\mathbf{Z}}$, respectively, where $\Omega$ is the subobject classifier of $\mc{E}$. A constant $e$ of type $E$ is interpreted by an arrow $\mathbf{1}\xrightarrow{\mathbf{e}} \mathbf{E}$ where $\mathbf{E}$ is the interpretation of $E$. For a term $\tau$ and an object $\X$ of $\mc{E}$, we write $\tau\p \X$ to mean $\tau$ has a type $X$ interpreted by the object $\X$. 

Let $\phi(x_1,...,x_n)$ be a formula with variables $x_1 \p \X_1,...,x_n\p \X_n$. Let $c_1 \in \X_j(C),...,c_n \in \X_n(C)$ for some object $C$ of $\mc{C}$. We define the relation $C$ \emph{forces} $\phi(x_1,...,x_n)[c_1,...,c_n]$ written $C \forces \phi(x_1,...,x_n)[c_1,...,c_n]$ by induction on the structure of $\phi$. 

\begin{defn}[Forcing]
\label{forcing}
First we replace the constants in $\phi$ by variables of the same type as follows: Let $e_1\p\mathbf{E}_1,...,e_m\p\mathbf{E}_m$ be the constants in $\phi(x_1,...,x_n)$ then $C\forces \phi(x_1,...,x_n)[c_1,...,c_n]$ iff \begin{equation*} C \forces \phi[y_1/e_1,...,y_m/e_m] (y_1,...,y_m,x_1,...,x_n) [\mathbf{e}_{1_C}(\ast),...,\mathbf{e}_{m_C}(\ast),c_1,...,c_n]\end{equation*} 
where $y_i\p \mathbf{E}_i$ and $\mathbf{e}_i:\mathbf{1}\rightarrow \mathbf{E}_i$ is the interpretation of $e_i$.

Now it suffices to define the forcing relation for formulas free of constants by induction as follows:

\begin{enumerate}[resume]
\item[\fbox{$\top$}] $C \forces \top$.
\item[\fbox{$\bot$}] $C \forces \bot $ iff the empty family is a cover of $C$.
\item[\fbox{$=$}] $C \forces (x_1 = x_2) [c_1,c_2]$ iff $c_1 = c_2$. 
\item[\fbox{$\land$}] $C \forces (\phi \land \psi)(x_1,...,x_n)[c_1,...,c_n]$ iff  $C \forces \phi(x_1,...,x_n)[c_1,...,c_n]$ and $C \forces \psi(x_1,...,x_n)[c_1,...,c_n]$.
\item[\fbox{$\lor$}] $C \forces (\phi \lor \psi)(x_1,...,x_n)[c_1,...,c_n]$ iff there exist a cover $\{C_i \xrightarrow{\;f_i\;} C\}_{i\in I} \in \bJ^\ast(C)$ such that \\$C_i \forces \phi(x_1,...,x_n)[c_1 f_i,...,c_n f_i]$ or $C_i \forces \psi(x_1,...,x_n)[c_1 f_i,...,c_n f_i]$ for each $i\in I$.
\item[\fbox{$\Rightarrow$}] $C \forces (\phi\Rightarrow \psi)(x_1,...,x_n)[c_1,...,c_n]$ iff  for every morphism $f:D \rightarrow C$ whenever\\ $D\forces \phi(x_1,...,x_n)[c_1 f,...,c_n f]$ one has $D\forces \psi(x_1,...,x_n)[c_1 f,...,c_n f]$.
\end{enumerate}
Let $y$ be a variable of the type $Y$ interpreted by the object $\Y$ of $\mc{E}$.
\begin{enumerate}[resume]
\item[\fbox{$\exists$}] $C\forces (\exists y \phi(x_1,...,x_n,y)) [c_1,...,c_n]$ iff there exist a cover $\{C_i \xrightarrow{\;f_i\;} C\}_{i\in I} \in \bJ^\ast(C)$ such that for each $i\in I$ one has $C_i\forces \phi(x_1,...,x_n,y)[c_1 f_i,...,c_n f_i,d]$ for some $d \in \Y(C_i)$. 
\item[\fbox{$\forall$}] $C \forces (\forall y \phi(x_1,...,x_n,y)) [c_1,...,c_n]$ iff  for every morphism $f:D\rightarrow C$ and for all $d \in \Y(D)$ one has $D\forces \phi(x_1,...,x_n,y)[c_1 f,...,c_n f,d]$.
\end{enumerate}
\end{defn}

We have the following derivable \emph{local character} and \emph{monotonicity} laws:
\begin{enumerate}[resume]
\item [\fbox{LC}] If $\{C_i \xrightarrow{f_i} C\}_{i \in I} \in \bJ^\ast(C)$ and for all $i\in I$, $C_i \forces \phi(x_1,...,x_n)[c_1 f_i,...,c_n f_i]$ then $C \forces \phi(x_1,...,x_n)[c_1,...,c_n]$. 
\item [\fbox{M}] If $C \forces \phi(x_1,...,x_n)[c_1,...,c_n]$ and $f:D\rightarrow C$ then $D \forces \phi(x_1,...,x_n)[c_1 f,...,c_n f]$.
\end{enumerate}
\section{The topos $\Sh(\raop,\bJ)$}
\label{site}

\begin{defn}[Regular ring]
A commutative ring $R$ is \emph{(von Neumann)} regular if for every element $a \in R$ there exist $b\in R$ 
such that  $a b a = a$ and $b a b = b$. This element $b$ is called the quasi-inverse of $a$. 
\end{defn}

 The quasi-inverse of an element $a$ is unique for $a$ \cite [Ch. 4] {lombardi_book}. We thus use the notation $a^*$ to refer to the quasi-inverse of $a$. A ring is regular iff it is zero-dimensional and reduced. To be regular is equivalent to the fact that
any principal ideal (consequently, any finitely generated ideal) is generated by an idempotent.
If $R$ is regular and $a\in R$ then $e=aa^*$ is an idempotent such that
$\ideal{e}=\ideal{a}$ and $R$ is isomorphic to $R_0\times R_1$ with $R_0=R/\ideal{e}$ and
$R_1 = R/\ideal{1-e}$. Furthermore $a$ is $0$ on the component $R_0$ and invertible on the component
$R_1$.

\begin{defn}[Fundamental system of orthogonal idempotents]
A family $(e_i)_{i \in I}$ of idempotents in a ring $R$ is a fundamental system of orthogonal idempotents if $\sum_{i \in I} e_i = 1$ and $\forall i, j [i \neq j \Rightarrow e_i e_j = 0]$.
\end{defn}

\begin{lem}
Given a fundamental system of orthogonal idempotents $(e_i)_{i\in I}$ in a ring $A$ we have a decomposition $A \cong \prod_{i\in I} A/\ideal{1-e_i}$.
\end{lem}
\begin{proof}
Follows by induction from the fact that $A \cong A/\ideal{e} \times A/\ideal{1-e}$ for an idempotent $e \in A$.
\end{proof}

\begin{defn}[Separable polynomial]
Let $R$ be a ring. A polynomial $p \in R[X]$ is separable if there exist $r,s\in R[X]$ such that $r p + s p' = 1$, where $p' \in R[X]$ is the derivative of $p$.
\end{defn}

\begin{defn}
A ring $R$ is a (strict) B\'ezout ring if for all $a, b \in R$ we can find $g, a_1, b_1, c, d \in R$ such that
$a = a_1 g$, $b =b_1 g$ and $c a_1+ d b_1= 1$ \cite [Ch. 4] {lombardi_book}.
\end{defn}

 If $R$ is a regular ring then $R[X]$ is a strict B\'ezout ring (and the converse is true \cite{lombardi_book}).
Intuitively we can compute the gcd as if $R$ was a field, but we may need to split $R$ when deciding
if an element is invertible or $0$. Using this, we see that given $a,b$ in $R[X]$ we can find a decomposition
$R_1,\dots ,R_n$ of $R$ and for each $i$ we have
 $g, a_1, b_1, c, d$ in $R_i[X]$ such that
$a = a_1 g$, $b =b_1 g$ and $c a_1+ d b_1= 1$ with $g$ monic. 

\begin{lem}\label{extension}
If $R$ is regular and $p$ in $R[X]$ is a separable polynomial then $R[a]=R[X]/\ideal{p}$ is regular.
\end{lem}

\begin{proof}
If $c = q(a)$ is an element of $R[a]$ with $q$ in $R[X]$ we compute the gcd $g$ of $p$ and $q$.
If $p = gp_1$, we can find $u$ and $v$ in $R[X]$ such that $ug + vp_1 = 1$ since $p$ is separable.
We then have $g(a) p_1(a) = 0$ and $u(a)g(a) + v(a)p_1(a) = 1$. It follows that $e = u(a)g(a)$ is
idempotent and we have $\ideal{e} = \ideal{g(a)}$.
\end{proof}

An algebra $A$ over a field $K$ is \emph{finitely presented} if it is of the form $K[X_1,..,X_n]/\ideal{f_1,...,f_m}$, i.e. the quotient of the polynomial ring over $K$ in finitely many variables by a finitely generated ideal. 

In order to build the classifying topos of a coherent theory $T$ it is customary in the literature to consider the category of all finitely presented $T_0$ algebras where $T_0$ is an equational subtheory of $T$. The axioms of $T$ then give rise  to a coverage on the dual category  \cite[Ch. 9]{makkai1977first}. For our purpose consider the category $\mc{C}$ of finitely presented $K$-algebras. Given an object $R$ of $\mc{C}$, the axiom schema of algebraic closure and the field axiom give rise to families
\begin{enumerate*}[label=(\roman*.)]
\item $R \rightarrow R[X]/\ideal{p}$ where $p\in R[X]$ is monic and
\item \begin{tikzpicture}[baseline=(current bounding box.west)]
  \node[outer sep=0] {\begin{tikzcd}[column sep=tiny,row sep=tiny] {}&R/\ideal{a} \\ R \arrow{ur} \arrow{dr} \\ {}& R[\tfrac{1}{a}] \end{tikzcd}};\end{tikzpicture}, for $a \in R$.
\end{enumerate*} 
Dualized, these are elementary covering families of $R$ in $\op{\mc{C}}$. We observe however that we can limit our consideration only to those finitely presented $K$-algebras that are zero dimensional and reduced, i.e. regular. In this case we can assume $a$ is an idempotent and we only consider extensions $R[X]/\ideal{p}$ where $p$ is separable. 

Let $\ra$ be the small category of finitely presented regular algebras over a fixed field $K$ and $K$-homomorphisms. First we fix an countable set of names $S$. An object of $\ra$ is a regular algebra of the form $K[X_1,...,X_n]/\ideal{f_1,...,f_m}$ where $X_i \in S$ for all $1\leq i \leq n$. Note that for any object $R$, there is a unique morphism $K \rightarrow R$.  A finitely presented regular $K$-algebra $A$ is a finite dimensional $K$-algebra, i.e. $A$ has a finite dimension as a vector space over $K$ \cite[Ch 4, Theorem 8.16]{lombardi_book}. The trivial ring $0$ is the terminal object in the category $\ra$ and $K$ is its initial object. 
  
To specify a coverage $\bJ$ on the category $\raop$, we define for each object $A$ a collection $\op{\bJ}(A)$ of families of morphisms of $\ra$ with domain $A$. We then take $\bJ(A)$ to be the dual of $\op{\bJ}(A)$ in the sense that $\{\overline{\varphi_i}: A_i \rightarrow A\}_{i\in I} \in \bJ(A)$ \emph{if and only if} $\{\varphi_i: A \rightarrow A_i\}_{i\in I} \in \op{\bJ}(A)$ where $\varphi_i$ of $\ra$ is the dual of $\overline{\varphi}_i$ of $\raop$. We call $\op{\bJ}$ cocoverage. We call an element of $\op{\bJ}(A)$ an elementary cocover (cocovering family) of $A$.  We define $\op{\bJ^\ast}$ similarly. We call elements of $\op{\bJ^\ast}(A)$ cocovers (cocovering families) of $A$. By a \emph{separable extension} of a ring $R$ we mean a ring $R[a] = R[X]/\ideal{p}$ where $p\in R[X]$ is non-constant, monic and separable.

\begin{defn}[Topology for $\raop$] 
\label{top}
For an object $A$ of $\ra$ the cocovering families are given by:
\begin{enumerate}[label*=(\roman*.),align=left]
\item \label{genDecomp} If $(e_i)_{i\in I}$ is a fundamental system of orthogonal idempotents of $A$, then $\{A \xrightarrow{\;\varphi_i\;} A/\ideal{1-e_i}\}_{i \in I} \in \op{\bJ}(A)$
where for each $i\in I$, $\varphi_i$ is the canonical homomorphism.
\item \label{genEmbed} Let $A[a]$ be a separable extension of $A$. We have $\{A \xrightarrow{\vartheta} A[a]\} \in \op{\bJ}(A)$ where $\vartheta$ is the canonical embedding.
%\item \label{genConcat} If $\{A \xrightarrow{\;\varphi_i\;} A_i\}_{i\in I} \in \op{\bJ}(A)$ and for each $i\in I$, $\{A_i\xrightarrow{\vartheta_{ij}} A_{ij}\}_{j\in J_i}\in \op{\bJ}(A_i)$ then $\{A\xrightarrow{\vartheta_{ij}\varphi_i} A_{ij}\}_{i\in I, j\in J_i} \in \op{\bJ}(A)$.
\end{enumerate}
\end{defn}

Note that in particular \ref{top}.\ref{genDecomp} implies that the trivial algebra $0$ is covered by the empty family of morphisms since an empty family of elements in this ring form a fundamental system of orthogonal idempotents. Also note that \ref{top}.\ref{genEmbed} implies that $\{A \xrightarrow{\;1_A\;} A\}\in \op{\bJ}(A)$.

\begin{lem}
The function $\bJ$ of Definition \ref{top} is a coverage on $\raop$.
\end{lem}
\begin{proof}
Let $\eta:R \rightarrow A$ be a morphism of $\ra$ and $S\in \op{\bJ}(R)$. We show that there exist an elementary cocover $T\in\op{\bJ}(A)$ such that for each $\vartheta \in T$, $\vartheta \eta$ factors through some $\varphi \in S$. By duality, this implies $\bJ$ is a coverage on $\raop$. By case analysis on the clauses of Definition \ref{top}.
\begin{enumerate}[label=(\roman*.),align=left]
\item If $S = \{\varphi_i : R \rightarrow R/\ideal{1-e_i}\}_{i \in I}$, where $(e_i)_{i\in I}$ is a fundamental system of orthogonal idempotents of $R$. In $A$, the family $(\eta(e_i))_{i\in I}$ is fundamental system of orthogonal idempotents. We have an elementary cocover $\{\vartheta_i: A \rightarrow A/\ideal{1-\eta(e_i)}\}_{i \in I }\in \op{\bJ}(A)$. For each $i\in I$,  the homomorphism $\eta$ induces a $K$-homomorphism $\eta_{e_i}:R/\ideal{1-e_i} \rightarrow A/\ideal{1-\eta(e_i)}$ where $\eta_{e_i}(r+	\ideal{1-e_i}) = \eta(r) + \ideal{1-\eta(e_i)}$. Since $\vartheta_i(\eta( r)) = \eta(r) + \ideal{1-\eta(e_i)}$ we have that $\vartheta_i \eta=\eta_{e_i} \varphi_i$.

\item If $S = \{\varphi: R \rightarrow R[r]\}$ with $R[r]=R[X]/\ideal{p}$ and $p\in R[X]$ monic, non-constant, and separable. Since $s p + t p' =1$,  we have $\eta(s) \eta(p) + \eta(t) \eta(p') = \eta(s) \eta(p) + \eta(t) \eta(p)' = 1$. Then $q = \eta(p)\in A[X]$ is separable. Let $A[a]=A[X]/\ideal{q}$. We have an elementary cocover $\{\vartheta:A \rightarrow A[a]\} \in \op{\bJ}(A)$ where $\vartheta$ is the canonical embedding. Let $\zeta:R[r] \rightarrow A[a]$ be the $K$-homomorphism such that $\zeta|_R = \eta$ and $\zeta(r) = a$. For $b \in R$, we have $\vartheta(\eta(b)) = \zeta(\varphi(b))$.

%\item Let $S=\{\xi_{ij}\gamma_i: R\rightarrow R_{ij}\}_{i\in I,j\in J_i}$ with $\{\gamma_i:R\rightarrow R_i\}_{i\in I} \in \op{\bJ}(R)$ and for each $i\in I$, $\{\xi_{ij}:R_i \rightarrow R_{ij}\}_{j\in J_i} \in \op{\bJ}(R_i)$. By induction we obtain a cocover $\{\vartheta_v: A \rightarrow A_v\}_{v\in V} \in  \op{\bJ}(A)$ such that for each $v\in V$, $\vartheta_v \eta = \pi_{iv} \gamma_i$ for some $i\in I$ and $\pi_{iv}:R_i \rightarrow A_v$. From this, by induction, we get for each $v\in V$ a cocover $\{\sigma_{v\ell} : A_v \rightarrow A_{v\ell}\}_{v\in V, \ell \in L_v} \in  \op{\bJ}(A_{v})$ such that for each $\ell \in L_v$, $\sigma_{v\ell} \pi_{iv}$ factor through some $\xi_{ij}$. But  $\sigma_{v\ell} \vartheta_v \eta = \sigma_{v\ell} \pi_{iv} \gamma_i$ for some $i$, hence, $\sigma_{v\ell} \vartheta_v \eta$ factors through some $\xi_{ij} \gamma_i$. \qedhere
\end{enumerate}
\end{proof}

%%%%%%%%%%%%%LEM Topology is subcanonical%%%%%%%%%%%%%%%%%%
\begin{lem}
\label{splitfamilycompatible}
Let $\bP:\ra \rightarrow \catset$ be a presheaf on $\raop$ such that $\bP(0) = 1$. Let $R$ be an object of $\ra$ and let $(e_i)_{i\in I}$ be a fundamental system of orthogonal idempotents of $R$. For each $i\in I$, let $R_i = R/\ideal{1-e_i}$ and let $\varphi_i:R \rightarrow R_i$ be the canonical homomorphism. Any family $\{s_i\in \bP(R_i)\}$ is compatible.
\end{lem}
\begin{proof}
Let $B$ be an object and for some $i,j\in I$ let $\vartheta:R_i\rightarrow B$ and $\zeta:R_j \rightarrow B$ be such that $\vartheta \varphi_i = \zeta \varphi_j$. We will show that $\bP(\vartheta)(s_i)  = \bP(\zeta)(s_j)$.
\begin{enumerate}[label=(\roman*.),align=left]
\item If $i=j$, then since $\varphi_i$ is surjective we have $\vartheta = \zeta$ and $\bP(\vartheta) = \bP(\zeta)$.

\item If $i\neq j$, then since $e_i e_j = 0$, $\varphi_i(e_i) = 1$ and $\varphi_j(e_j) = 1$ we have $\varphi_j(e_i) = \varphi_j(e_i e_j) =  0$. But then 
\begin{align*}
1 = \vartheta(1) = \vartheta(\varphi_i(e_i)) = \zeta(\varphi_j(e_i)) = \zeta(0) = 0
\end{align*}
Hence $B$ is the trivial algebra $0$. By assumption $\bP(0) = 1$, hence $\bP(\vartheta)(s_i) = \bP(\zeta)(s_j) = \ast$. \qedhere
\end{enumerate}
\end{proof}

\begin{cor}
\label{sheafofprodequalprodofsheaf}
Let $\F$ be a sheaf on $(\raop,\bJ)$. Let $R$ be an object of $\ra$ and $(e_i)_{i\in I}$ a fundamental system of orthogonal idempotents of $R$. Let $R_i = R/\ideal{1-e_i}$ and $\varphi_i:R\rightarrow R_i$ be the canonical homomorphism. The map $f:\F(R) \rightarrow \prod_{i\in I} \F(R_i)$ such that $f(s)=(\F(\varphi_i) s)_{i \in I}$ is an isomorphism.
\end{cor}
\begin{proof}
Since $\F(0) = 1$, by Lemma \ref{splitfamilycompatible} any family $\{s_i \in \F(R_i)\}_{ i \in I}$ is compatible. Since $\F$ is a sheaf, the family $\{s_i \in \F(R_i)\}_{i \in I}$ has a unique amalgamation $s \in \F(R)$ with restrictions $s\varphi_i = s_i$.  The isomorphism is given by $f s = (s \varphi_i)_{i\in I}$. We can then use the tuple notation $(s_i)_{i\in I}$ to denote the element $s$ in $\F(R)$.
\end{proof}

One say that a polynomial $f \in R[X]$ has a \emph{formal degree} $n$ if $f$ can be written as $f = a_n X^n + ... + a_0$ which is to express that for any $m > n$ the coefficient of $X^m$ is known to be $0$. 

\begin{lem}
\label{separablehasnonequalroots}
Let $R$ be a regular ring and $p_1,p_2 \in R[X]$ be monic polynomials of degrees $n_1$ and $n_2$ respectively. Let $R[a,b]=R[X,Y]/\ideal{p_1(X),p_2(Y)}$. Let $q_1,q_2 \in R[Z]$ be of formal degrees $m_1 < n_1$ and $m_2 < n_2$ respectively. If $q_1(a) = q_2(b)$ then $q_1 =  q_2 = r \in R$.
\end{lem}
\begin{proof}
The statement follows immediately since the $R$-basis $a^i, i>0$ and $b^j, j>0$ are linearly independent.
%Let $q_1(a) = q_2(b)$, then in $R[X,Y]$ one has $q_1(X) - q_2(Y) = f(X,Y) p_1(X) + g(X,Y) p_2(Y)$ for some $f,g \in R[X,Y]$. In $R[a][Y] = R[X,Y]/\ideal{p_1(X)}$ we have $q_1(a) - q_2(Y) = g(a,Y) p_2(Y)$. But $p_2(Y)$ is monic of $Y$-degree $n_2$ while $q_2(Y)-q_1(a)$ has formal $Y$-degree $m_2 < n_2$, hence, the coefficients of $g(a,Y) \in R[a][Y]$ are all equal to $0$ in $R[a]$. We have then that all coefficient of $Y^\ell$ with $\ell > 0$ in $q_2(Y)$ are equal $0$. That is, $q_2 = r \in R$ and that $q_1(a)$ is equal to the constant coefficient $r$ of $q_2(Y)$. Thus in $R[X]$ we have $q_1(X) - r = h(X) p_1(X)$ for some $h \in R[X]$. Similarly, since $(q_1(X) - r)$ has a formal $X$-degree $m_1$ and $p_1$ is monic of degree $n_1 > m_1$ we get that $q_1 = r \in R$.
\end{proof}

\begin{cor}
\label{pushoutequalizerofextension}
Let $R$ be an object of $\ra$ and $p \in R[X]$ separable and monic. Let $R[a] = R[X]/\ideal{p}$ and $\varphi:R\rightarrow R[a]$ the canonical morphism. Let $R[b,c] = R[X,Y]/\ideal{p(X),p(Y)}$. The commuting diagram

\begin{tikzpicture}[baseline=(current bounding box.west)]
  \node[outer sep=0] {\begin{tikzcd}
R[a] \arrow{r}{\vartheta}                       & R[b,c] \\
R \arrow{u}{\varphi} \arrow{r}{\varphi} & R[a] \arrow{u}{\zeta} 
\end{tikzcd}};\end{tikzpicture}
$\vartheta|_R = \zeta|_R = 1_R$, $\vartheta(a) = b$, $\zeta(a) = c$

is a pushout diagram of $\ra$. Moreover, $\varphi$ is the equalizer of $\zeta$ and $\vartheta$.
\end{cor}
\begin{proof}
Let \begin{tikzpicture}[inner sep=0,baseline=-0.5ex]
  \node[outer sep=0] {\begin{tikzcd}[column sep=small] {R[a]\;} \arrow[shift left=.75ex]{r}{\eta} \arrow[shift right=.75ex]{r}[below]{\varrho} & {\;B} \end{tikzcd}};\end{tikzpicture} be morphisms of $\ra$ such that $\eta \varphi = \varrho \varphi$. Then for all $r\in R$ we have $\eta(r) = \varrho(r)$. Let $\gamma:R[b,c] \rightarrow B$ be the homomorphism such that $\gamma(r) = \eta(r) =\varrho(r)$ for all $r \in R$ while $\gamma(b) = \eta(a), \gamma(c) = \varrho(a)$. Then $\gamma$ is the unique map such that $\gamma \vartheta = \eta$ and $\gamma  \zeta = \varrho$.
 
Let $A$ be an object of $\ra$ and let $\varepsilon : A \rightarrow R[a]$ be a map such that $\zeta \varepsilon = \vartheta \varepsilon$. By Lemma \ref{separablehasnonequalroots} if for some $f \in R[a]$ one has $\zeta(f) = \vartheta(f)$ then $f \in R$ (i.e. $f$ is of degree $0$ as a polynomial in $a$ over $R$). Thus $\varepsilon(A) \subset R$ and we can factor $\varepsilon$ uniquely (since $\varphi$ is injective) as $\varepsilon = \varphi \mu$ with $\mu:A\rightarrow R$.
\end{proof}

Let $\{\varphi: R \rightarrow R[a]\}$ be a singleton elementary cocover. Since one can form the pushout of $\varphi$ with itself, the compatibility condition on a singleton family $\{s \in \F(R[a])\}$ can be simplified as: Let \begin{tikzpicture}[inner sep=0, baseline=-0.6ex]
  \node[outer sep=0] {\begin{tikzcd}[column sep=small]{R\;} \arrow{r}{\varphi} & {\;R[a]\;}\arrow[shift left=.75ex]{r}{\eta} \arrow[shift right=.75ex,swap]{r}{\vartheta} & {\;A}\end{tikzcd}};\end{tikzpicture} be a pushout diagram. A family $\{s \in \F(R[a])\}$ is compatible if and only if $s \vartheta = s \eta$.

\begin{cor}
The coverage $\bJ$ is subcanonical, i.e. all representable presheaves in $\catset^{\ra}$ are sheaves on $(\raop, \bJ)$.\qed
\end{cor}

%%%%%%%%%%%%%%%%%%%%%%%%%%%%%%%%%%%%%%%%%%%%%%%%%
\section{The algebraically closed field extension}
\label{algclosed}
We define the presheaf $\F: \ra \rightarrow \catset$ to be the forgetful functor. That is, for an object $A$ of $\ra$, $\F(A) = A$ and for a morphism $\varphi:A\rightarrow C$ of $\ra$, $\F(\varphi) = \varphi$. %Before we show that $\F$ is a sheaf we recall the following fact:
%%%%%%%%%%%%%LEM algebraic closure functor is a sheaf%%%%%%%%%%%%%%%%%%
%\begin{fact}
%Let $(\mc{C},\bT)$ be a site and $\bP:\op{\mc{C}}\rightarrow \catset$ a presheaf. If $\bP$ satisfies the sheaf axiom for a family $\{C_i\xrightarrow{\;f_i\;} C\}_{i\in I}$ and for each $i$, $\bP$ satisfies the sheaf axiom for a family $\{C_{ij}\xrightarrow{\;g_{ij}\;} C_i\}_{i\in I, j\in J_i}$, then $\bP$ satisfies the sheaf axiom for the family $\{C_{ij}\xrightarrow{\; f_i g_{ij}\;} C\}_{i\in I, j\in J_i}$, \textup{(See \cite[C2.1, Lemma 2.1.7]{elephant2})}.
%\end{fact}
%
%Thus in order to show that a presheaf on $\ra$ is a sheaf it will suffice to check the sheaf condition for cocovers of the forms \ref{top}.\ref{genDecomp} and \ref{top}.\ref{genEmbed} only.

\begin{lem}
\label{isasheaf}
$\F$ is a sheaf of sets on the site $(\raop,\bJ)$
\end{lem}
\begin{proof}
By case analysis on the clauses of Definition \ref{top}. 
\begin{enumerate}[label=(\roman*.),align=left]
\item Let $\{R \xrightarrow{\varphi_i} R/\ideal{1-e_i}\}_{i\in I} \in \op{\bJ}(R)$, where $(e_i)_{i\in I}$ is fundamental system of orthogonal idempotents of $R$.  The presheaf $\F$ has the property $\F(0) = 1$. By Lemma \ref{splitfamilycompatible} a family $\{a_i \in R/\ideal{1-e_i}\}_{ i \in I}$ is a compatible family. By the isomorphism $R \xrightarrow{(\varphi_i)_{i\in I}} \prod_{i\in I} R/\ideal{1-e_i}$ the element $a=(a_i)_{i\in I} \in R$ is the unique element such that $\varphi_i (a) = a_i$.

\item Let $\{R\xrightarrow{\varphi}R[a]\}\in \op{\bJ}(R)$ where $R[a] = R[X]/\ideal{p}$ with $p\in R[X]$ monic, non-constant and separable polynomial. Let $\{r \in  R[a]\}$ be a compatible family. Let \begin{tikzpicture}[inner sep =0, baseline=-0.35ex]
  \node[outer sep=0] {\begin{tikzcd}[column sep=small] {R\;}\arrow{r}{\varphi} & {\;R[a]\;} \arrow[shift left=.75ex]{r}{\vartheta} \arrow[shift right=.75ex]{r}[below]{\zeta} & {\;R[b,c]} \end{tikzcd}};\end{tikzpicture} be the pushout diagram of Corollary \ref{pushoutequalizerofextension}. Compatibility then implies $\vartheta(r) = \zeta(r)$ which by the same Corollary is true only if the element $r$ is in $R$. We then have that $r$ is the unique element restricting to itself along the embedding $\varphi$. \qedhere
\end{enumerate}
\end{proof}

We fix a field $K$ of characteristic $0$. Let $\mc{L}[F,+,.]$ be a language with basic type $F$ and function symbols $+,.:F\times F \rightarrow F$. We extend $\mc{L}[F,+,.]$ by adding a constant symbol of type $F$ for each element $a\in K$, to obtain $\mc{L}[F,+,.]_K$. Define $\Diag(K)$ as : if $\phi$ is an atomic $\mc{L}[F,+,.]_K$-formula or the negation of one such that $K\models \phi(a_1,...,a_n)$ then $\phi(a_1,...,a_n) \in \Diag(K)$. The theory $T$ equips the type $F$ with axioms of the geometric theory of algebraically closed field containing $K$

\begin{defn}
\label{ACFaxioms} The theory $T$ has the following sentences (with all the variables having the type $F$).
\begin{enumerate}
\item $\Diag(K)$.
\item The axioms of a commutative group:
   \begin{enumerate*}
    \item $\forall x\; [0 + x = x + 0 = x]$
    \item $\forall x \forall y \forall z [x + (y + z) = (x + y) + z]$
    \item $\forall x \exists y [x+y = 0]$
    \item $\forall x \forall y[x + y = y + x]$
   \end{enumerate*}
\item The axioms of a commutative ring:
    \begin{enumerate*}
    \item $\forall x\; [x 1 = x]$
    \item $\forall x\; [x 0 = 0]$
    \item $\forall x \forall y [x y = y x]$\\
    \item $\forall x \forall y \forall z[x(yz) = (xy)z]$
    \item $\forall x \forall y \forall z [x(y+z) = xy + xz]$
    \end{enumerate*}
\item The field axioms:
   \begin{enumerate*}
   \item $1 \neq 0$.
   \item $\forall x[x = 0 \lor \exists y[x y = 1]]$.
    \end{enumerate*}
\item The axiom schema for algebraic closure:
$\forall a_1\dots\forall a_n \exists x[x^n + \sum_{i=1}^{n} x^{n-i} a_i = 0]$.
\item $F$ is algebraic over $K$: $\forall x [\bigvee_{p\in K[Y]} p(x) = 0]$.
\end{enumerate} 
\end{defn}

With these axioms the type $F$ becomes the type of an algebraically closed field containing $K$. We proceed to show that with the interpretation of the type $F$ by the object $\F$ the topos $\Sh(\raop, \bJ)$ is a model of $T$, i.e. $\F$ is a model, in Kripke--Joyal semantics, of an algebraically closed field containing of $K$. First note that since there is a unique map $K \rightarrow C$ for any object $C$ of $\ra$, an element $a \in K$ gives rise to a unique map $\mathbf{1} \xrightarrow{a} \F$, that is the map $\ast \mapsto a \in \F(K)$. Every constant $a \in K$ of the language is then interpreted by the corresponding unique arrow $\mathbf{1} \xrightarrow{a} \F$. (we use the same symbol for constants and their interpretation to avoid cumbersome notation). That $\F$ satisfies $\Diag(K)$ then follows directly.

\begin{lem}
\label{isring}
$\F$ is a ring object.
\end{lem}
\begin{proof}
For an object $C$ of $\ra$ the object $\F(C)$ is a commutative ring.
\end{proof}

\begin{lem}
$\F$ is a field.
\end{lem}
\begin{proof}
For any object $R$ of $\ra$ one has $R\forces 1\neq 0$ since for any $R\xrightarrow{\varphi} C$ such that $C \forces 1 = 0$ one has that $C$ is trivial and thus $C\forces \bot$. Next we show that for variables $x$ and $y$ of type $\F$ and any object $R$ of $\raop$ we have $R\forces \forall x\; [x = 0 \lor \exists y\; [xy = 1]]$. Let $\varphi: A \rightarrow R$ be a morphism of $\raop$ and let $a \in A$. We need to show that $A \forces a = 0 \lor \exists y [ya= 1]$.  The element $e=aa^*$ is an idempotent and we have a cover $\{\varphi_1:  A/\ideal{e}\rightarrow A, \varphi_2:  A/\ideal{1-e}\rightarrow A\} \in \bJ^\ast(A)$ with 
$A/\ideal{e} \forces a \varphi_1   = 0$ and 
$ A/\ideal{1-e} \forces (a \varphi_2) (a^* \varphi_2 ) = e \varphi_2 = 1$. 
Hence by \fbox{$\exists$} we have $A/\ideal{1-e} \forces \exists y [(a \varphi_2) y = 1]$ and by $\fbox{$\lor$}$, $A/\ideal{1-e} \forces a \varphi_2 = 0\lor \exists y [(a \varphi_2) y = 1]$. Similarly, $A/\ideal{e}\forces a \varphi_1 = 0 \lor \exists y [(a\varphi_1) y = 1]$. By \fbox{$\forall$} we get $R\forces \forall x\; [x = 0 \lor \exists y\; [xy = 1]]$.
\end{proof}

To show that $A \forces \forall a_1\dots\forall a_n \exists x\;[x^n + \sum_{i=1}^{n} x^{n-i} a_i = 0]$ for every $n$, we need to be able to extend an algebra $R$ of $\ra$ with the appropriate roots. We need the following lemma.

\begin{lem}
\label{squarefree}
Let $L$ be a field and $f \in L[X]$ a monic polynomial. Let $g = \langle f,f'\rangle$, where $f'$ is the derivative of $f$. Writing $f = hg$ we have that $h$ is separable.
We call $h$ the separable associate of $f$.
\end{lem}
\begin{proof} Let $a$ be the gcd of $h$ and $h'$. We have $h = l_1 a$.
Let $d$ be the gcd of $a$ and $a'$. We have $a = l_2 d$ and $a' = m_2 d$, with $l_2$ and
$m_2$ coprime.

 The polynomial $a$ divides $h' = l_1a' + l_1'a$ and hence that $a = l_2 d$ divides
$l_1 a' = l_1 m_2 d$. It follows that $l_2$ divides $l_1m_2$ and since $l_2$ and $m_2$
are coprime, that $l_2$ divides $l_1$.

 Also, if $a^n$ divides $p$ then $p = q a^n$ and $p' = q'a^n + nqa'a^{n-1}$. Hence
$d a^{n-1}$ divides $p'$. Since $l_2$ divides $l_1$, this implies 
that $a^n = l_2 d a^{n-1}$ divides $l_1 p'$. So $a^{n+1}$ divides $al_1 p' = h p'$.

 Since $a$ divides $f$ and $f'$, $a$ divides $g$. We show that $a^n$ divides $g$ for all $n$ by induction on $n$. If $a^n$ divides $g$
we have just seen that $a^{n+1}$ divides $g'h$. Also $a^{n+1}$ divides $h'g$ since $a$ divides $h'$.
So $a^{n+1}$ divides $g'h+h'g = f'$. On the other hand, $a^{n+1}$ divides $f = hg = l_1ag$.
So $a^{n+1}$ divides $g$ which is the gcd of $f$ and $f'$.  This implies that $a$ is a unit.
\end{proof}

%The next lemma captures the intuition that the separable associate contains all the roots of the polynomial.
%\begin{lem}
%\label{squarefreedivides}
%Let $L$ be a field of characteristic $0$ and $f \in L[X]$ a nonzero monic polynomial of degree $n$. If $h$ is the separable associate of $f$, then $f \mid h^n$.
%\end{lem}
%\begin{proof}
%Let $g = \langle f, f' \rangle$ and $f = h g$. We prove $f \mid h^n$ by induction on the degree of $g$. If  $\deg g=0$, then $f \mid h$. 
%Let $f' = q g$ and $r h + s q = 1$. Assuming the statement holds for $\deg g \leq m$. Let $\deg g=m+1$. Let $g'$ be the derivative of $g$ we can find $l_1,l_2,r_1,r_2,d$ such that $g = l_1 d, g' = l_2 d, r_1 l_1 + r_2 l_2 = 1$. Since $\deg d < \deg g = m+1$, by induction hypothesis we get that $g \mid l_1^{m+1}$. Let $g e = l_1^{m+1}$. We have $f' = q g = h' g + h g'$, substituting we get $h' l_1 d + h l_2 d = q l_1 d$. We know that $d \neq 0$, hence $h' l_1 + h l_2 = q l_1$. From this and $r_1 l_1 + r_2 l_2 = 1$ we get $r_2 h' l_1 + h - h r_1 l_1 = r_2 q l_1$. Thus $l_1(r_2 q + h r_1 - r_2 h')= h$. If we let $c=r_2 q + h r_1 - r_2 h'$ and since $f = h g$, we have $f e c^{m+1} = h (g e) c^{m+1} = h l_1^{m+1} c^{m+1} = h^{m+2}$. 
%\end{proof}

Since $\F$ is a field, the previous lemma holds for polynomials over $\F$. This means that for all objects $R$ of $\raop$ we have $R\forces \textup{Lemma }\textup{\ref{squarefree}}$. Thus we have the following Corollary.

\begin{cor}
\label{squarefreereg}
Let $R$ be an object of $\ra$ and let $f$ be a monic polynomial of degree $n$ in $R[X]$ and $f'$ its derivative. There is a cocover $\{\varphi_i:R\rightarrow R_i\}_{i\in I} \in \op{\bJ^{\ast}}(R)$ and for each $R_i$ we have $h,g,q,r,s\in R_i[X]$ such that $\varphi_i (f) = h g$, $\varphi_i (f') = q g$ and $r h + s q = 1$. Moreover, $h$ is monic and separable.\qed
\end{cor}

Note that in characteristic $0$, if $f$ is monic and non-constant the separable associate of $f$ is non-constant.
\begin{lem}
The field object $\F \in \Sh(\raop, \bJ)$ is algebraically closed.
\end{lem}
\begin{proof}
We prove that for all $n>0$ and all $(a_1,...,a_n)\in \F^n(R) = R^n$, one has $R \forces \exists x \;[x^n +\sum_{i=1}^{n} x^{n-i} a_i = 0]$.
Let $f = x^n +\sum_{i=1}^{n} x^{n-i} a_i$. By Corollary \ref{squarefreereg} we have a cover $\{\vartheta_j:R_j \rightarrow R\}_{j\in I} \in \bJ^\ast(R)$ such that in each $R_j$ we have $g = \langle f \vartheta_j, f' \vartheta_j\rangle$ and $f\vartheta_j = h g$ with $h \in R_j[X]$ monic and separable. Note that if $\deg f \geq 1$, $h$ is non-constant. For each $R_j$ we have a singleton cover $\{\varphi:R_j[b] \rightarrow R_j\mid R_j[b]=R_j[X]/\ideal{h}\} \in \bJ^\ast(R_j)$. That is, we have $R_j[b] \forces b^n + \sum_{i=1}^{n} b^{n-1}(a_i\vartheta_j \varphi)=0$. By \fbox{$\exists$} we get $R_j[b] \forces \exists x\;[x^n + \sum_{i=1}^{n} x^{n-1}(a_i\vartheta_j \varphi)=0]$ and by \fbox{LC} we have $R_j \forces \exists x\;[x^n + \sum_{i=1}^{n} x^{n-1}(a_i\vartheta_j)=0]$. Since this is true for each $R_j, j\in J$ we have by \fbox{LC} $R\forces \exists x\;[x^n + \sum_{i=1}^{n} x^{n-1} a_i=0]$.
\end{proof}

\begin{lem}
$\F$ is algebraic over $K$.
\end{lem}
\begin{proof}
We will show that for any object $R$ of $\ra$ and element $r \in R$ one has $R\forces \bigvee_{p\in K[X]} p(r) = 0$. Since $R$ is a finitely presented $K$-algebra we have that $R$ is a finite integral extension of a polynomial ring $K[Y_1,...,Y_n]\subset R$ where $Y_1,..,Y_n$ are elements of $R$ algebraically independent over $K$ and that $R$ has Krull dimension $n$ \cite[Ch 13, Theorem 5.4] {lombardi_book}. Since $R$ is zero-dimensional (i.e. has Krull dimension $0$) we have $n=0$ and $R$ is integral over $K$, i.e. any element $r \in R$ is the zero of some monic polynomial over $K$.
\end{proof}

%%%%%%%%%%%%%%%%%%%%%%%%%%%%%The power series object%%%%%%%%%%%%%%%%%%%%%%%%%%%
\section{Constant sheaves, natural numbers, and power series}
\label{sec:powerseires}
Here we describe the object of natural numbers in the topos $\Sh(\raop, \bJ)$ and the object of power series over the field $\F$. This will be used in section \ref{AC} to show that the axiom of dependent choice does not hold when the base field $K$ is the rationals and later in the example of Newton--Puiseux theorem (section \ref{eliminating}).

Let $\bP:\ra \rightarrow \catset$ be a constant presheaf associating to each object $A$ of $\ra$ a discrete set $B$. That is, $\bP(A) = B$ and $\bP(A\xrightarrow{\varphi}R) = 1_B$ for all objects $A$ and all morphism $\varphi$ of $\ra$. Let $\Q:\ra \rightarrow \catset$ be the presheaf such that $\Q(A)$ is the set of elements of the form $\{(e_i,b_i)\}_{i\in I}$ where $(e_i)_{i\in I}$ is a fundamental system of orthogonal idempotents of $A$ and for each $i$, $b_i \in B$. We express such an element as a formal sum $\sum_{i\in I} e_i b_i$. Let $\varphi:A \rightarrow R$ be a morphism of $\ra$, the restriction of $\sum_{i\in I} e_i b_i \in \Q(A)$ along $\varphi$ is given by $(\sum_{i\in I} e_i b_i) \varphi = \sum_{i\in I}\varphi(e_i)  b_i \in \Q(R)$. In particular with canonical morphisms $\varphi_i:A \rightarrow A/\ideal{1-e_i}$, one has for any $j \in I$ that $(\sum_{i\in I} e_i b_i) \varphi_j = b_j \in \Q(A/\ideal{1-e_j})$. Two elements $\sum_{i\in I} e_i b_i \in \Q(A)$ and $\sum_{j\in J} d_j c_j \in \Q(A)$ are equal if and only if $\forall i\in I, j\in J [b_i \neq c_j \Rightarrow e_i d_j = 0]$. 	

To prove that $\Q$ is a sheaf we will need the following lemmas.
\begin{lem}
\label{amalgamationofidemp}
Let $R$ be a regular ring and let $(e_i)_{i\in I}$ be a fundamental system of orthogonal idempotents of $R$. Let $R_i = R/\ideal{1-e_i}$ and $([d_j])_{j\in J_i}$ be a fundamental system of orthogonal idempotents of $R_i$, where $[d_j] = d_j + \ideal{1-e_i}$. The family $(e_i d_j)_{i \in I, j \in J_i}$ is a fundamental system of orthogonal idempotents of $R$.
\end{lem}
\begin{proof}
In $R$ one has $\sum_{j \in J_i} e_i d_j = e_i \sum_{j\in J_i} d_j = e_i (1 + \ideal{1-e_i}) = e_i$. Hence, $\displaystyle\sum_{i\in I, j\in J_i} e_i d_j = \sum_{i\in I} e_i = 1$. For some $i\in I$ and $t, k \in J_i$ we have $(e_i d_t) (e_i d_k) = e_i (0 + \ideal{1-e_i}) = 0$ in $R$. Thus for $i, \ell \in I$, $j \in J_i$ and $s \in J_\ell$ one has $i \neq \ell \lor j \neq s \Rightarrow (e_i d_j) (e_\ell d_s) = 0$.
\end{proof}

\begin{lem}
\label{thegoodlemma}
Let $R$ be a regular ring, $f \in R[Z]$ a polynomial of formal degree $n$ and $p \in R[Z]$ a monic polynomial of degree $m > n$. If in $R[X,Y]$ one has $f(Y) (1-f(X)) = 0 \mod \ideal{p(X),p(Y)}$ then $f = e \in R$ with $e$ an idempotent.
\end{lem}
\begin{proof}
Let $f(Z) = \sum_{i = 0}^n r_i Z^i$. By the assumption, for some $q, g \in R[X,Y]$
\begin{equation*}
\begin{split}
f(Y) (1-f(X)) = \sum_{i=0}^n r_i (1-\sum_{j=0}^n r_j X^j) Y^i = qp(X) + g p(Y)
\end{split}
\end{equation*}
One has $\sum_{i=0}^n r_i (1-\sum_{j=0}^n r_j X^j) Y^i = g(X,Y) p(Y) \mod \ideal{p(X)}$. Since $p(Y)$ is monic of $Y$-degree greater than $n$, one has that $r_i (1-\sum_{j=0}^n r_j X^j) = 0 \mod \ideal{p(X)}$ for all $0 \leq i \leq n$. But this means that $r_i r_n X^n + r_i r_{n-1} X^{n-1} + ... + r_i r_0 - r_i$ is divisible by $p(X)$ for all $0 \leq i \leq n$ which because $p(X)$ is monic of degree $m > n$  implies that all coefficients are equal to $0$. In particular, for $1 \leq i \leq n$ one gets that $r_i^2 = 0$ and hence $r_i = 0$ since $R$ is reduced. For $i=0$ we have $r_0 r_0 - r_0 = 0$ and thus $r_0$ is an idempotent of $R$.
\end{proof}

\begin{lem}
The presheaf $\Q$ described above is a sheaf on $(\raop,\bJ)$.
\end{lem}
\begin{proof}
By case analysis on Definition \ref{top}.
\begin{enumerate}[label=(\roman*.),align=left]
\item Let $\{R\xrightarrow{\varphi_i} R/\ideal{1-e_i}\}_{i\in I} \in \op{\bJ}(R)$ where $(e_i)_{i \in I}$ be a fundamental system of orthogonal idempotents of an object $R$. Let $R/\ideal{1-e_i} = R_i$. Since $\Q(0) = 1$ by Lemma \ref{splitfamilycompatible} any set $\{s_i \in \Q(R_i)\}_{i\in I}$ is  compatible. For each $i$, Let $s_i = \sum_{j \in J_i} [d_j] b_j$. By Lemma \ref{amalgamationofidemp} we have an element $s=\displaystyle\sum_{i\in I, j \in J_i} (e_i d_j)  b_j \in \Q(R)$ the restriction of which along $\varphi_i$ is the element $\sum_{j \in J_i} [d_j] b_j \in \Q(R_i)$. It remains to show that this is the only such element. Let there be an element $\sum_{\ell \in L} c_\ell  a_\ell \in \Q(R)$ that restricts to $u_i = s_i$ along $\varphi_i$. We have $u_i = \sum_{\ell \in L} [c_\ell] a_\ell$. One has that for any $j\in J_i$ and $\ell \in L$, $b_j \neq a_\ell \Rightarrow [c_\ell d_j] = 0$ in $R_i$, hence, in $R$ one has $b_j \neq a_\ell \Rightarrow c_\ell d_j = r (1-e_i)$. Multiplying both sides of $c_\ell d_j = r (1-e_i)$ by $e_i$ we get $b_j \neq a_\ell \Rightarrow c_\ell (e_i d_j)= 0$. Thus proving $s=\sum_{\ell \in L} c_\ell  a_\ell$.

\item Let $\{\varphi:R\rightarrow R[a]=R[X]/\ideal{p}\}\in \op{\bJ}(R)$ where $p \in R[X]$ is monic non-constant  and separable. Let the singleton $\{s=\sum_{i \in I} e_i b_i \in \Q(R[a]) \}$ be compatible. We can assume w.l.o.g. that $\forall i,j \in I\; [ i\neq j \Rightarrow b_i \neq b_j]$ since if $b_k = b_\ell$ one has that $(e_k+e_\ell) b_l  + \sum_{j \in I}^{j \neq \ell, j \neq k} e_j b_j = s$. (Note that an idempotent $e_i$ of $R[a]$ is a polynomial $e_i(a)$ in $a$ of formal degree less than $\deg p$). Let $R[c,d] = R[X,Y]/\ideal{p(X),p(Y)}$, by Corollary \ref{pushoutequalizerofextension}, one has a pushout diagram 
\begin{tikzpicture}[inner sep=0, baseline=-0.8ex]
  \node[outer sep=0] {\begin{tikzcd}[column sep=small]
{R\;} \arrow{r}{\varphi} & {\;R[a]\;} \arrow[shift left=.75 ex]{r}{\zeta} \arrow[shift right=0.75 ex,swap]{r}{\vartheta} & {\;R[c,d]}
\end{tikzcd}};\end{tikzpicture} where $\zeta|_R = \vartheta|_R = 1_R$, $\zeta(a) = d$ and $\vartheta(a)=c$. That the singleton $\{s\}$ is compatible then means $s \vartheta = \sum_{i\in I} e_i(c) b_i  = s \zeta = \sum_{i\in I}  e_i(d) b_i$, i.e. $\forall i,j\in I\;[b_i \neq b_j \Rightarrow e_i(c) e_j(d)= 0]$. By the assumption that $b_i \neq b_j$ whenever $i\neq j$ we have in $R[c,d]$ that $e_j(d) e_i(c) = 0$ for any $i\neq j \in I$. Thus $e_j(d) \sum_{i\neq j} e_i(c)  = e_j(d) (1-e_j(c)) =0$, i.e. in $R[X,Y]$ one has $e_j(Y) (1-e_j(X)) =0 \mod \ideal{p(X),p(Y)}$. By Lemma \ref{thegoodlemma} we have that $e_j(X)=e_j(Y) = e \in R$. We have thus shown $s$ is equal to $\sum_{j\in J} d_j b_j \in \Q(R[a])$ such that $d_j \in R$ for $j \in J$. That is $\sum_{j\in J} d_j b_j \in \Q(R)$. Thus we have found a unique (since $\Q(\varphi)$ is injective) element in $\Q(R)$ restricting to $s$ along $\varphi$. \qedhere
\end{enumerate}
\end{proof}

\begin{lem}
\label{constantsheaf}
Let $\bP$ and $\Q$ be as described above. Let $\Gamma:\bP \rightarrow \Q$ be the presheaf morphism such that $\Gamma_R(b) = b \in \Q(R)$ for any object $R$ and $b \in B$. If $\E$ is a sheaf and $\Lambda:\bP \rightarrow \E$ is a morphism of presheaves, then there exist a unique sheaf morphism $\Delta:\Q \rightarrow \E$ such that the following diagram, of $\catset^\ra$, commutes.
\begin{tikzpicture}[baseline=(current bounding box.west)]
  \node[outer sep=0] {\begin{tikzcd}
\bP \arrow{r}{\Lambda} \arrow{d}{\Gamma} & \E \\
\Q \arrow[swap,dotted]{ru}{\Delta}
\end{tikzcd}};\end{tikzpicture}
That is to say, $\Gamma:\bP \rightarrow \Q$ is the sheafification of $\bP$.
\end{lem}
\begin{proof}
Let $a =\sum_{i\in I} e_i b_i \in \Q(A)$ and let $A_i = A/\ideal{1-e_i}$ with canonical morphisms $\varphi_i: A \rightarrow A_i$. 

Let $\E$ and $\Lambda$ be as in the statement of the lemma. If there exist a sheaf morphism $\Delta:\Q \rightarrow \E$, then $\Delta$ being a natural transformation forces us to have for all  $i\in I$, $ \E(\varphi_i) \Delta_A = \Delta_{A_i} \Q(\varphi_i)$. By Lemma \ref{sheafofprodequalprodofsheaf}, we know that the map $d\in \E(A) \mapsto (\E(\varphi_i) d \in \E(A_i))_{i\in I}$ is an isomorphism. Thus it must be that $\Delta_A(a) = (\Delta_{A_i} \Q(\varphi_i)(a) )_{i\in I} = (\Delta_{A_i}(b_i))_{i\in I}$. But $\Delta_{A_i}(b_i) = \Delta_{A_i} \Gamma_{A_i}(b_i)$. To have $\Delta \Gamma = \Lambda$  we must have $\Delta_{A_i}(b_i)=\Lambda_{A_i}(b_i)$. Hence, we are forced to have $\Delta_A(a) = (\Lambda_{A_i}(b_i))_{i\in I}$. Note that $\Delta$ is unique since its value $\Delta_A(a)$ at any $A$ and $a$ is forced by the commuting diagram above.
\end{proof}

The constant presheaf of natural numbers $\N$ is the natural numbers object in $\catset^\ra$. We associate to $\N$ a sheaf $\sN$ as described above. From Lemma \ref{constantsheaf} one can easily show that $\sN$ satisfy the axioms of a natural numbers object in $\Sh(\raop,\bJ)$.

\begin{defn}
Let $\F[[X]]$ be the presheaf mapping each object $R$ of $\ra$ to $\F[[X]](R) =R[[X]]= R^\nats$ with the obvious restriction maps.
\end{defn} 

\begin{lem}
$\F[[X]]$ is a sheaf.
\end{lem}
\begin{proof}
The proof is immediate as a corollary of Lemma \ref{isasheaf}.
\end{proof}

\begin{lem}
\label{powerserieslemma}
The sheaf $\F[[X]]$ is naturally isomorphic to the sheaf $\F^{\sN}$.
\end{lem}
\begin{proof}
Let $C$ be an object of $\raop$. Since $\F^{\sN}(C) \cong \y_C \times \sN \rightarrow \F$, an element $\alpha_C \in \F^{\sN}(C)$ is a family of elements of the form $\alpha_{C,D}:\y_C(D) \times \sN(D) \rightarrow \F(D)$ where $D$ is an object of $\raop$. 
Define $\Theta:\F^{\sN} \rightarrow \F[[X]]$ as $(\Theta \alpha)_C (n) = \alpha_{C,C} (1_C, n)$. Define $\Lambda:\F[[X]]\rightarrow \F^{\sN}$ as 
\begin{align*}
(\Lambda\beta)_{C,D}(C\xrightarrow{\varphi} D, \sum_{i\in I} e_i n_i) = (\vartheta_i \varphi (\beta_C (n_i)))_{i\in I} \in \F(D)
\end{align*}
where $D \xrightarrow{\vartheta_i} D/\ideal{1-e_i}$ is the canonical morphism. Note that by Lemma \ref{sheafofprodequalprodofsheaf} one indeed has that $(\vartheta_i \varphi (\beta_C (n_i)))_{i\in I} \in \prod_{i\in I} \F(D_i) \cong \F(D)$. One can easily verify that $\Theta$ and $\Lambda$ are natural. It remains to show the isomorphism. One one hand we have
\begin{equation*}
\begin{split}
(\Lambda \Theta \alpha)_{C,D} (\varphi, \sum_{i\in I} e_i n_i)  & = (\vartheta_i \varphi ((\Theta \alpha)_C (n_i)))_{i\in I}
= (\vartheta_i \varphi (\alpha_{C,C} (1_C, n_i)))_{i\in I} \\
& = ((\alpha_{C,D_i} (\vartheta_i \varphi ,n_i)))_{i\in I} = \alpha_{C,D} (\varphi, \sum_{i\in I} e_i n_i)
\end{split}
\end{equation*}
Thus showing $\Lambda \Theta = 1_{\F^{\sN}}$. On the other hand, $(\Theta \Lambda \beta)_C (n)  = (\Lambda \beta)_{C,C}(1_C, n)= 1_C 1_C (\beta_C(n)) = \beta_C(n)$. Thus $\Theta \Lambda = 1_{\F[[X]]}$.
\end{proof}

\begin{lem}
The power series object $\F[[X]]$ is a ring object.
\end{lem}
\begin{proof}
A Corollary to Lemma \ref{isring}.
\end{proof}

%%%%%%%%%%%%%%%%%%%%%%%%%%%%%%%%%%%Choice Axiom%%%%%%%%%%%%%%%%%%%%%%%%%%%%%%%%%
\section{Choice axioms}
\label{AC}
The (\emph{external}) axiom of choice fails to hold (even in a classical metatheory) in the topos $\Sh(\raop,\bJ)$ whenever the field $K$ is not algebraically closed. To show this we will show that there is an epimorphism in $\Sh(\raop,\bJ)$ with no section. 

\begin{fact}
\label{epifact}
Let $\Theta: \bP \rightarrow \G$ be a morphism of sheaves on a site $(\mc{C},\bJ)$. Then $\Theta$ is an epimorphism if for each object $C$ of $\mc{C}$ and each element $c \in \G(C)$ there is a cover $S$ of $C$ such that for all $f:D\rightarrow C$ in the cover $S$ the element $c f$ is in the image of $\Theta_D$. \textup{\cite[Ch. 3]{sheavesgeometry}.}
\end{fact}

\begin{lem}
Let $K$ be a field of characteristic $0$ not algebraically closed. There is an epimorphism in $\Sh(\raop,\bJ)$ with no section.
\end{lem}
\begin{proof}
Let $f=X^n + \sum_{i=1}^n r_i X^{n-i}$ be a non-constant polynomial for which no root in $K$ exist. w.l.o.g. we assume $f$ separable. One can construct $\Lambda: \F \rightarrow \F$ defined by $\Lambda_C(c) = c^n + \sum_{i=1}^{n-1} r_i c^{n-i} \in C$. 
Given $d\in \F(C)$, let $g= X^n + \sum_{i=1}^{n-1} r_i X^{n-i}-d$. By Corollary \ref{squarefreereg} there is a cover $\{C_\ell \xrightarrow{\;\varphi_\ell\;} C\}_{\ell\in L} \in \bJ^\ast(C)$ with $h_\ell \in C_\ell[X]$ a separable non-constant polynomial dividing $g$. Let $C_\ell[x_\ell] = C_\ell[X]/\ideal{h_\ell}$ one has a singleton cover $\{C_\ell[x_\ell]\xrightarrow{\vartheta_\ell} C_\ell\}$ and thus a composite cover $\{C_\ell[x_\ell]\xrightarrow{\vartheta_\ell \varphi_\ell} C\}_{\ell \in L}\in \bJ^\ast(C)$. Since $x_\ell$ is a root of $h_\ell \mid g$ we have $\Lambda_{C_\ell[x_\ell]}(x_\ell) = x_\ell^n + \sum_{i=1}^{n-1} r_i x_\ell^{n-i} = d$ or more precisely $\Lambda_{C_\ell[x_\ell]}(x_\ell) = d \varphi_\ell \vartheta_\ell$. Thus, $\Lambda$ is an epimorphism (by Fact \ref{epifact}) and it has no section, for if it had a section $\Psi:\F \rightarrow \F$ then one would have $\Psi_K(-r_n) = a \in K$ such that $a^n + \sum_{i=1}^n r_i a^{n-i}= 0$ which is not true by assumption.
\end{proof}

\begin{thm}
\label{choicefails}
Let $K$ be a field of characteristic $0$ not algebraically closed. The axiom of choice fails to hold in the topos $\Sh(\raop,\bJ)$.\qed
\end{thm}

We note that in Per Martin-L\"of type theory one can show that (see \cite{martinloftypes})
\begin{align*}
&(\prod x\in A)(\sum y\in B[x])C[x,y] \Rightarrow (\sum f \in (\prod x \in A) B[x]) (\prod x\in A) C[x,f(x)]
\end{align*}
As demonstrated in the topos $\Sh(\raop,\bJ)$ we have an example of an intuitionistically valid formula of the form $\forall x \exists y \phi(x,y)$ where no  function $f$ exist for which $\exists f \forall x \phi(x,f(x))$ holds.

 We demonstrate further that when the base field is $\rationals$ the weaker axiom of \emph{dependent choice} does not hold (internally) in the topos $\Sh(\raqop,\bJ)$. For a relation $R \subset Y\times Y$ the axiom of dependent choice is stated as
\begin{equation}
\tag{ADC}
\forall x \exists y R(x,y) \Rightarrow \forall x \exists g\in Y^N [g(0) = x \land \forall n R(g(n),g(n+1))]
\end{equation}
\begin{thm}
\label{dependentchoicefails}
$\Sh(\raqop,\bJ)\forces \neg \textup{ADC}$.
\end{thm}
\begin{proof}
Consider the binary relation on the algebraically closed object $\F$ defined by the characteristic function $\phi(x,y) := y^2-x =0$. Assume $C \forces \textup{ADC}$ for some object $C$ of $\ra$. Since $C\forces \forall x \exists y [y^2-x = 0]$ we have 
$C\forces \forall x \exists g\in \F^{\sN} [g(0) = x \land \forall n [g(n)^2=g(n+1)]]$.
That is for all morphisms $C\xrightarrow{\zeta} A$ of $\ra$ and elements $a \in \F(A)$ one has
$
A\forces \exists g\in \F^{\sN} [g(0) = a \land \forall n [g(n)^2 = g(n+1)]]
$.
Taking $a = 2$ we have $A\forces \exists g\in \F^{\sN} [g(0) = 2 \land \forall n [g(n)^2 = g(n+1)]]$. Which by \fbox{$\exists$} implies the existence of a cocover $\{\eta_i:A \rightarrow A_i\mid i\in I\}$ and power series $\alpha_i \in \F^{\sN}(A_i)$ such that
$
A_i\forces \alpha_i(0) = 2 \land \forall n [\alpha_i(n)^2 = \alpha_i(n+1)]]
$.
By Lemma \ref{powerserieslemma} we have $\F^{\sN}(A_i) \cong A_i[[X]]$ and thus the above forcing implies the existence of a series $\alpha_i=2 + 2^{1/2} + ... + 2^{1/2^j} + ... \in  A_i[[X]]$. But this holds only if $A_i$ contains a root of $X^{2^j} - 2$ for all $j$ which implies $A_i$ is trivial as will shortly show after the following remark.

Consider an algebra $R$ over $\rationals$. Assume $R$ contains a root of $X^{2^n} - 2$ for some $n$. Then letting $\rationals[x]=\rationals[X]/\ideal{X^{2^n} - 2}$, one will have a homomorphism $\xi: \rationals[x] \rightarrow R$. By Eisenstein's criterion the polynomial $X^{2^n} - 2$ is irreducible over $\rationals$, making $\rationals[x]$ a field of dimension $2^n$ and $\xi$ either an injection with a trivial kernel or $\xi = \rationals[x]\rightarrow 0$. 

Now we continue with the proof. Until now we have shown that for all $i\in I$, the algebra $A_i$ contains a root of $X^{2^j} - 2$ for all $j$. For  each $i\in I$,  let $A_i$ be of dimension $m_i$ over $\rationals$. We have that $A_i$ contains a root of $X^{2^{m_i}} - 2$ and we have a homomorphism $\rationals(\sqrt[2^{m_i}]{2}) \rightarrow A_i$ which since $A_i$ has dimension $m_i < 2^{m_i}$ means that $A_i$ is trivial for all $i \in I$. Hence, $A_i \forces \bot$ and consequently $C \forces \bot$. We have shown that for any object $D$ of $\raqop$ if $D\forces \textup{ADC}$ then $D\forces \bot$. Hence $\Sh(\raqop,\bJ)\forces \neg \textup{ADC}$.
\end{proof}
As a consequence we get that the \emph{internal} axiom of choice does not hold in $\Sh(\raqop,\bJ)$. %Alternatively we can demonstrate this by showing that while  $\textup{Sqr}:\F\rightarrow \F$ defined by $\textup{Sqr}_C(c) = c^2$ is an epimorphism, the induced morphism $\textup{Sqr}^\F: \F^\F \rightarrow \F^\F$ is not an epimorphism.

%%%%%%%%%%%%%%%%%%%%%%%%%%%%%%%%%Elimination%%%%%%%%%%%%%%%%%%%%%%%%%%%%%%%%%%%%%%%%%%

\section{Eliminating the algebraic closure assumption}
\label{eliminating}

Let $K$ be a field of characteristic $0$. We consider a typed language $\mc{L}[N,F]_K$ of the form described in Section \ref{prel} with two basic types $N$ and $F$ and the elements of the field $K$ as its set of constants. Consider a theory $T$ in the language $\mc{L}[N,F]_K$, such that $T$ has as an axiom every atomic formula or the negation of one valid in the field $K$, $T$ equips $N$ with the (Peano) axioms of natural numbers and equips $F$ with the axioms of a field containing $K$.  If we interpret the types $N$ and $F$ by the objects $\sN$ and $\F$, respectively, in the topos $\Sh(\raop,\bJ)$ then we have, by the results proved earlier, a model of $T$ in $\Sh(\raop,\bJ)$. Let $\mathrm{AlgCl}$ be the axiom schema of algebraic closure with quantification over the type $F$, then one has that $T + \mathrm{AlgCl}$ has a model in $\Sh(\raop,\bJ)$ with the same interpretation. Let $\phi$ be a sentence in the language such that $T + \mathrm{AlgCl} \vdash \phi$ in IHOL deduction system. By soundness \cite{awodey97logicin} one has that $\Sh(\raop,\bJ) \forces \phi$, i.e. for all finite dimensional regular algebras $R$ over $K$, $R\forces \phi$ which is then a constructive interpretation of the existence of the algebraic closure of $K$.

This model can be implemented, e.g. in Haskell. In the paper \cite{dynnewton} by the authors, an algorithm for computing the Puiseux expansions of an algebraic curve based on this model is given. The statement with the assumption of algebraic closure is:

\textit{``
Let $K$ be a field of characteristic $0$ and $G(X,Y)=Y^n + \sum_{i=1}^n b_i(X) Y^{n-i} \in K[[X]][Y]$  a monic, non-constant polynomial separable over $K((X))$. Let $F$ be the algebraic closure of $K$, we have a positive integer $m$ and a factorization $G(T^m,Y) = \prod_{i=1}^n (Y-\alpha_i)$ with $\alpha_i \in F[[T]]$ "}

We can then extract the following computational content

\textit{``
Let $K$ be a field of characteristic $0$ and $G(X,Y)=Y^n + \sum_{i=1}^n b_i(X) Y^{n-i} \in K[[X]][Y]$  a monic, non-constant polynomial separable over $K((X))$. Then there exist a (von Neumann) regular algebra $R$ over $K$ and a positive integer $m$  such that $G(T^m,Y) = \prod_{i=1}^n (Y-\alpha_i)$ with $\alpha_i \in R[[T]]$
"}

For example applying the algorithm to $G(X,Y) = Y^4 - 3 Y^2 + X Y + X^2 \in \rationals[X,Y]$ we get a regular algebra $\rationals[b,c]$ with $b^2-13/36=0$ and $c^2-3=0$ and a factorization
\begin{align*}
& G(X,Y) =\\
& \;\; (Y+(-b-\tfrac{1}{6})X+(-\tfrac{31}{351}b-\tfrac{7}{162})X^3+(-\tfrac{1415}{41067}b-\tfrac{29}{1458})X^5+...) \\
& \;\;   (Y+(b-\tfrac{1}{6})X+(\tfrac{31}{351}b-\tfrac{7}{162})X^3+(\tfrac{1415}{41067}b-\tfrac{29}{1458})X^5+...) \\
&\;\;   (Y-c+\tfrac{1}{6}X+\tfrac{5}{72}cX^2+\tfrac{7}{162}X^3+\tfrac{185}{10368}cX^4+\tfrac{29}{1458}X^5+...)\\
&\;\;   (Y+c+\tfrac{1}{6}X-\tfrac{5}{72}cX^2+\tfrac{7}{162}X^3-\tfrac{185}{10368}cX^4+\tfrac{29}{1458}X^5+...)
\end{align*}

Another example of a possible application of this model is as follows: suppose one want to show that 

\textit{``For discrete field $K$, if $f \in K[X,Y]$ is smooth, i.e. $1\in \ideal{f, f_x, f_Y}$, then $K[X,Y]/\ideal{f}$ is a Pr\"{u}fer ring.``}

To prove that a ring is Pr\"{u}fer one needs to prove that it is arithmetical, that is $\forall x,y \exists u, v, w[yu = vx \land yw=(1-u)x]$. Proving that $K[X,Y]/\ideal{f}$ is arithmetical is easier in the case where $K$ is algebraically closed \cite{curvesprufer}. Let $\F$ be the algebraic closure of $K$ in $\Sh(\raop,\bJ)$. Now $\F[X,Y]/\ideal{f}$ being arithmetical amounts to having a solution $u$,$v$, and $w$ to a linear system $yu = vx$, $yw=(1-u)x$. Having obtained such solution, by Rouch\'{e}--Capelli--Fonten\'{e} theorem we can conclude that the system have a solution in $K[X,Y]/\ideal{f}$.
%%%%%%%%%%%%%%%%%%%%%%%%%%%%%%Boolean%%%%%%%%%%%%%%%%%%%%%%%%%%%%%%%
\section{The logic of $\Sh(\raop, \bJ)$}
\label{sec:booleaness}
In this section we will demonstrate that in a \emph{classical metatheory} one can show that the topos $\Sh(\raop, \bJ)$ is boolean. In fact we will show that, in a classical metatheory, the boolean algebra structure of the subobject classifier is the one specified by the boolean algebra of idempotents of the algebras in $\ra$. Except for Theorem \ref{notshownboolean} the reasoning in this section is classical. Recall that the idempotents of a commutative ring form a boolean algebra with the meaning of the logical operators given by : $\top = 1$, $\bot = 0$, $e_1 \land e_2 = e_1 e_2$, $e_1 \lor e_2 = e_1 + e_2 - e_1 e_2$ and $\neg e = 1-e$. We write $e_1 \leq e_2$ iff $e_1 \land e_2 = e_1$ and $e_1 \lor e_2 = e_2$

A sieve $S$ on an object $C$ is a set of morphisms with codomain $C$ such that if $g \in S$ and $\codom{h} = \dom{g}$ then $g h \in S$. A cosieve is defined dually to a sieve. A sieve $S$ is said to cover a morphism $f:D\rightarrow C$ if $f^*(S)=\{g\mid \codom{g}=D, fg \in S\}$ contains a cover of $D$. Dually, a cosieve $M$ on $C$ is said to cover a morphism $g:C\rightarrow D$ if the sieve dual to $M$ covers the morphism dual to $g$. 

\begin{defn}[Closed cosieve]
A sieve $M$ on an object $C$ of $\mc{C}$ is closed if for all $f$ with $\codom{f} = C$ if $M$ covers $f$ then $f \in M$. A closed cosieve on an object $C$ of $\op{\mc{C}}$ is the dual of a closed sieve in $\mc{C}$.
\end{defn}

\begin{fact}[Subobject classifier]
The subobject classifier in the category of sheaves on a site $(\mc{C},\bJ)$ is the presheaf $\Omega$ where for an object $C$ of $\mc{C}$ the set $\Omega(C)$ is the set of closed sieves on $C$ and for each $f: D\rightarrow C$ we have a restriction map $M \mapsto \{h \mid \codom{h} = D, fh \in M\}$.
\end{fact}

\begin{lem}
\label{closedsievefield}
Let $R$ be an object of $\ra$. If $R$ is a field the closed cosieves on $R$ are the maximal cosieve $\{f\mid \dom{f}=R\}$ and the minimal cosieve $\{R \rightarrow 0\}$.
\end{lem}
\begin{proof}
Let $S$ be a closed cosieve on $R$ and let $\varphi:R\rightarrow A \in S$ and let $I$ be a maximal ideal of $A$. If $A$ is nontrivial we have a field morphism $R\rightarrow A/I$ in $S$ where $A/I$ is a finite field extension of $R$. Let $A/I = R[a_1,...,a_n]$	. But then the morphism $\vartheta: R\rightarrow R[a_1,...,a_{n-1}]$ is covered by $S$. Thus $\vartheta \in S$ since $S$ is closed. By induction on $n$ we get that a field automorphism $\eta:R\rightarrow R$ is in $S$ but then by composition of $\eta$ with its inverse we get that $1_R \in S$. Consequently, any morphism with domain $R$ is in $S$.
\end{proof}

\begin{cor}
For an object $R$ of $\ra$. If $R$ is a field, then $\Omega(R)$ is a 2-valued boolean algebra.
\end{cor}
\begin{proof}
This is a direct Corollary of Lemma \ref{closedsievefield}. The maximal cosieve $(1_R)$ correspond to the idempotent $1$ of $R$, that is the idempotent $e$ such that, $\ker 1_R = \ideal{1-e}$. Similarly the cosieve $\{R\rightarrow 1\}$ correspond to the idempotent $0$.
\end{proof}

\begin{cor} 
For an object $A$ of $\ra$, $\Omega(A)$ is isomorphic to the set of idempotents of $A$ and the Heyting algebra structure of $\Omega(A)$ is the boolean algebra of idempotents of $A$.
\end{cor}
\begin{proof}
Classically a finite dimension regular algebra over $K$ is isomorphic to a product of field extensions of $K$. Let $A$ be an object of $\ra$, then $A \cong F_1\times...\times F_n$ where $F_i$ is a finite field extension of $K$. The set of idempotents of $A$ is $\{(d_1,...,d_n) \mid 1\leq j \leq n,  d_j \in F_j, d_j = 0 \text{ or } d_j = 1\}$. But this is exactly the set $\Omega(F_1)\times...\times \Omega(F_n) \cong \Omega(A)$. It is obvious that since $\Omega(A)$ is isomorphic to a product of boolean algebras, it is a boolean algebra with the operators defined pointwise.
\end{proof}

\begin{thm}
\label{booleaness}
The topos $\Sh(\raop,\bJ)$ is boolean.
\end{thm}
\begin{proof}
The subobject classifier of $\Sh(\raop,\bJ)$ is $1\xrightarrow{\true} \Omega$ where for an object $A$ of $\ra$ one has $\true_A(\ast) = 1 \in A$.
\end{proof}

It is not possible to show that the topos $\Sh(\raop,\bJ)$ is boolean in an intuitionistic metatheory as we shall demonstrate. First we recall the definition of the \emph{Limited principle of omniscience} (LPO for short).
\begin{defn}[LPO]
For any binary sequence $\alpha$ the statement $\forall n [\alpha(n) = 0] \lor \exists n [\alpha(n) = 1]$ holds.
\end{defn}
LPO cannot be shown to hold intuitionistically. One can, nevertheless, show that it is weaker than the law of excluded middle \cite{bridges1987varieties}.

\begin{thm}
\label{notshownboolean}
Intuitionistically, if $\Sh(\raop,\bJ)$ is boolean then \textup{LPO} holds.
\end{thm}
\begin{proof}
Let $\alpha \in K[[X]]$ be a binary sequence. By Lemma \ref{powerserieslemma} one has an isomorphism $\Lambda:\F[[X]] \isoarrow \F^{\sN}$. Let $\Lambda_K(\alpha) = \beta \in \F^{\sN}(K)$. Assume the topos $\Sh(\raop,\bJ)$ is boolean. Then one has 
$
K\forces \forall n[\beta(n) = 0] \lor \exists n[\beta(n) = 1]
$.
 By \fbox{$\lor$} this holds only if there exist a cocover of $K$
\begin{align*}
\{\vartheta_i: K \rightarrow A_i \mid i \in I\}\cup\{\xi_j: K\rightarrow B_j \mid j\in J\}
\end{align*}
such that $B_j \forces \forall n [(\beta\xi_j)(n) = 0]$ for all $j\in J$ and $A_i \forces \exists n[(\beta\vartheta_i)(n) = 1]$ for all $i\in I$. Note that at least one of $I$ or $J$ is nonempty since $K$ is not covered by the empty cover.

For each $i\in I$ there exist a cocover $\{\eta_\ell:A_i \rightarrow D_\ell \mid \ell \in L\}$ of $A_i$ such that for all $\ell\in L$,  we have $D_\ell \forces (\beta\vartheta_i \eta_\ell)(m) = 1$ for some $m \in \sN(D_\ell)$. Let $m =\sum_{t\in T} e_t n_t$ then we have a cocover $\{\xi_t:D_\ell \rightarrow C_t=D_\ell/\ideal{1-e_t} \mid t \in T\}$ such that $C_t \forces (\beta\vartheta_i \eta_\ell \xi_t) (n_t) = 1$ which implies $\xi_t \eta_\ell \vartheta_i (\alpha(n_t)) = 1$. For each $t$ we can check whether $\alpha(n_t) = 1$. If $\alpha(n_t) = 1$ then we have witness for $\exists n [\alpha(n) = 1]$. Otherwise, we have $\alpha(n_t) = 0$ and $\xi_t \eta_\ell \vartheta_i(0) = 1$. Thus the map $\xi_t \eta_\ell \vartheta_i : K \rightarrow C_t$ from the field $K$ cannot be injective, which leaves us with the conclusion that $C_t$ is trivial. If for all $t \in T$, $C_t$ is trivial then $D_\ell$ is trivial as well. Similarly, if for every $\ell\in L$, $D_\ell$ is trivial then $A_i$ is trivial as well. At this point one either have either
\begin{enumerate*}[label=(\roman*)] 
\item a natural number $m$ such that $\alpha(m) = 1$ in which case we have a witness for $\exists n [\alpha(n) = 0]$.
\end{enumerate*} Or 
\begin{enumerate*}[label=(\roman*),resume]
\item we have shown that for all $i \in I$, $A_i$ is trivial in which case we have a cocover $\{\xi_j:K\rightarrow B_j \mid j\in J\}$ such that $B_j \forces \forall n [(\beta\xi_j)(n) = 0]$ for all $j\in J$. Which by \fbox{LC} means $K \forces \forall n [\beta(n) = 0]$ which by \fbox{$\forall$} means that for all arrows $K\rightarrow R$ and elements $d \in \sN(R)$, $R\forces \beta(d) = 0$. In particular for the arrow $K\xrightarrow{1_K} K$ and every natural number $m$ one has $K \forces \beta(m) = 0$ which implies $K\forces \alpha(m) = 0$. By \fbox{$=$} we get that $\forall m \in \nats[\alpha(m) = 0]$.
\end{enumerate*}
Thus we have shown that LPO holds. 
\end{proof}

\begin{cor}
It cannot be shown in an intuitionistic metatheory that the topos $\Sh(\raop,\bJ)$ is boolean.\qed
\end{cor}

%%%%%%%%%%%%%%%%%%%%%%%%%%%%%%BIBLIOGRAPHY%%%%%%%%%%%%%%%%%%%%%%%%%%%
\bibliographystyle{eptcs}
\bibliography{/Users/bmannaa/st/bibDb}
\end{document}